\documentclass[english]{amsart}%
\usepackage{amsfonts}
\usepackage{amsmath}
\usepackage{amssymb}
\usepackage{amsthm}
\usepackage{amscd}
\usepackage{babel}
\usepackage[latin1]{inputenc}
\usepackage{graphicx}%
\providecommand{\U}[1]{\protect\rule{.1in}{.1in}}
\newtheorem{teor}{Theorem}
\newtheorem{cor}{Corollary}
\newtheorem{prop}{Proposition}

\theoremstyle{definition}
\newtheorem{exa}{Example}
\theoremstyle{definition}

\renewcommand{\subjclassname}{AMS \textup{2010} Mathematics Subject
Classification\ }
\email{grau@uniovi.es}
\begin{document}
\author{L. Bayón}
\address{Departamento de Matemáticas, Universidad de Oviedo\\
Avda. Calvo Sotelo s/n, 33007 Oviedo, Spain}
\email{bayon@uniovi.es}
\author{P. Fortuny}
\address{Departamento de Matemáticas, Universidad de Oviedo\\
Avda. Calvo Sotelo s/n, 33007 Oviedo, Spain}
\email{fortunypedro@uniovi.es}
\author{J. Grau}
\address{Departamento de Matemáticas, Universidad de Oviedo\\
Avda. Calvo Sotelo s/n, 33007 Oviedo, Spain}
\email{grau@uniovi.es}
\author{A. M. Oller-Marcén}
\address{Centro Universitario de la Defensa de Zaragoza - IUMA\\
Ctra. Huesca s/n, 50090 Zaragoza, Spain}
\email{oller@unizar.es}
\author{M. M. Ruiz}
\address{Departamento de Matemáticas, Universidad de Oviedo\\
Avda. Calvo Sotelo s/n, 33007 Oviedo, Spain}
\email{mruiz@uniovi.es}

\title{The multi-returning secretary problem}

\begin{abstract}
In this paper we consider the so-called Multi-returning secretary problem, a version of the Secretary problem in which each candidate has $m$
identical copies. The case $m=2$ has already been completely solved by several authors using different methods both the case $m>2$ had not been satisfactorily solved yet. Here, we provide and efficient algorithm to compute the optimal threshold and the probability of success for every $m$. Moreover, we give a method to determine their asymtoptic values based on the solution of a system of $m$ ODEs. 
\end{abstract}

\maketitle

\subjclassname{60G40, 62L15}

\section{Introduction}

The so-called \emph{Secretary problem} is possibly one of the most famous problem in
optimal stopping theory. This problem can be stated as follows: We want to select the best out $n$ rankable candidates. The candidates are inspected one by one at random order and we have to accept or reject the candidate immediately. At each step, we can rank the candidate among all the preceding ones, but we are unaware of the quality of yet unseen candidates. The goal is to determine the
optimal strategy that maximizes the probability of selecting the best candidate.

Dynkin \cite{48} and Lindley \cite{101} independently proved that the best strategy consists in rejecting roughly the first $n/e$ interviewed candidates and then selecting the first one that is better
than all the preceding ones. Following this strategy, the probability of
selecting the best candidate is at least $1/e$, this being its approximate
value for large values of $n$. This well-known solution was later refined by
Gilbert and Mosteller \cite{gil}, showing that $\left\lfloor (n-\frac{1}%
{2})e^{-1}+\frac{1}{2}\right\rfloor $ is a better approximation than $\lfloor
n/e\rfloor$, although the difference is never greater than 1. The Secretary problem has been addressed by many authors in different fields
such as applied probability, statistics and decision theory. Extensive
bibliographies on the topic can be found in \cite{FER}, \cite{FER2} and
\cite{2009} for instance.

This classical problem has several modifications that can be addressed and solved in a rather straightforward
manner using the so-called \emph{odds-algorithm} devised by Bruss \cite{ods}. We can mention, for example:
\begin{itemize}
\item The \emph{Best or Worst problem}, in which the goal is to select either the best or the worst candidate \cite{nuestropos}.
\item The \emph{Postdoc problem}, in which the goal is to select the second best candidate \cite{posdoc,nuestropos} or, even more generally, the $k$-th best candidate \cite{aesima}.
\item The \emph{Win, Lose or Draw marriage problem}, in which payoff is $1$ if
the player select the best, $-1$ if he select someone who is not the best, and 0
if no object is selected and the goal is to maximize the payoff \cite{saka}.
\item The \emph{Secretary problem with uncertain employment}, in which a candidate may refuse to be accepted with a given probability \cite{refusal}.
\item The \emph{One of the best two problem}, in which the goal is to select either the best or the second best candidate receiving different payoffs in each case \cite{gil, gus}.
\end{itemize}

Another interesting variant of the classical problem, the so-called \emph{Returning Secretary problem} was introduced by Garrod in 2012 \cite{tesis, artigarrod}. In this variant, every candidate has an identical copy and the goal is still to select the best candidate. In 2015, Shai Vardi \cite{gemelasvardi, arx} independently addressed the same problem. Both Garrod and Vardi approach the problem from the perspective of partially ordered sets. Very recently, Grau \cite{yo} introducing a new methodology based on the
solution of differential equations.

The previous variant can be further generalized to consider the \emph{Multi-returning Secretary problem}, in which every candidate has $m$ identical copies or, equivalently, in which every candidate is inspected $m$ times. Garrod \cite{tesis} shows that the optimal strategy in the Multi-returning Secretary problem is a threshold strategy (just like in the classical problem). He also provides explicit formulas for the optimal threshold $\mathbf{k}_n^m$ \cite[Theorem 2.2]{tesis} and for the probability of success $\mathbf{P}_n^m$ \cite[Theorem 2.16]{tesis}. However, his formulas are very inefficient from the computational point of view. In fact, for a fixed $m$, the formula for $\mathbf{k}_n^m$ requires a number of operations of order $O(n^2)$ while the formula for $\mathbf{P}_n^m$ requires a number of operations of order $O(n^{m-1})$. Regarding the asymptotic behavior, Garrod is able to prove for every fixed $n$ that $\lim_m \mathbf{k}_n^m=\lceil n/2\rceil$ and that $\lim_m \mathbf{P}_n^m=1$. However, the limitations of Garrod's approach for $m>2$ are clearly shown in the following table \cite[p. 51]{tesis}.

\begin{table}[h]\label{tablagarrod}
\caption{Table from \cite[p.51]{tesis}. Numbers in boldface are wrong in the original (correct values are in parentheses).}
\begin{tabular}
[c]{|l|l|l|l|l|l|}\hline
$m$ & $\mathbf{k}_{100}^m$ & $\mathbf{k}_{1000}^{m}$ &  
$\lim_{n}\left(  \frac{\mathbf{k}_{n}^{m}}{n}\right)  $ & $\mathbf{P}_{100}^{m}$ & $\lim_{n}\mathbf{P}_{n}^{m}$\\\hline
1 & 38 & 369 & 0.3679 & $\mathbf{0.3708}$ (0.371042)& 0.3678794 \\\hline
2 & 48 &  471 & 0.4709  & 0.76970661 &  0.7679742\\\hline
3 & 50 & 493  & ?  & $\mathbf{0.9354}$ (0.93518)& ? \\\hline
4 & 50 & 499  & ?   & ?  &  ? \\\hline
5 & 50 & 500  & ?  & ?  &   ?  \\\hline
6 & 50 & 500   & ?   & ?    & ?  \\\hline
7 & 50 & 500     & ?    & ?        &  ? \\\hline
8 & 50 &500     &?    & ?          &   ?\\\hline
9 & 50  &500    & ?  & ?  & ?  \\\hline
10 & 50  &500  & ?  &?   &  ? \\\hline
\end{tabular}
\end{table}

Motivated by these limitations, Garrod \cite[p. 52]{tesis} presents a series of open problems:
\begin{enumerate}
\item Improve the formula for $\mathbf{P}_n^m$.
\item For fixed ${m}$, prove that $\lim_n \left(\frac{\mathbf{k}_n^m}{n}\right)$ exists.
\item For fixed ${m}$, prove that $\lim_n \mathbf{P}_n^m$ exists.
\item If $\lim_n\left(\frac{\mathbf{k}_n^m}{n}\right)$ exists, find its value as the root of an equation or as a function of
$\lim_n\left(\frac{\mathbf{k}_n^{m-1}}{n}\right)$.
\item If $\lim_n\mathbf{P}_n^{m}$ exists, find its value as the root of an equation or as a function of
$\lim_n\mathbf{P}_n^{m-1}$.
\end{enumerate}

In the present paper we address the previous open problems. In particular, we give efficient algorithms that compute $\mathbf{k}_n^m$ and $\mathbf{P}_n^m$ as well as a method, based on the techniques introduced in \cite{yo} to compute their asymptotic values.

The paper is organized as follows. In Section 2 we present some technical results. Section 3 is devoted to revisit the ${m}$-returning secretary problem using a dynamic programming approach, providing a method to compute $\mathbf{P}_n^m$. In Sections 4 and 5, using the ideas from \cite{yo}, we give methods to compute the asympotic values of $\mathbf{k}_n^m$ and $\mathbf{P}_n^m$. Finally, Section 6 concludes the paper relating our results to Garrod's open problems.

\section{Some technical results}

In this section we present some technical results that will be used extensively in forthcoming sections. The first proposition was already introduced in \cite[Proposition 1]{yo} and, in some sense, it extends \cite[Proposition 1]{nuestropos}.

\begin{prop}
\label{conv} Let $\{F_{n}\}_{n\in\mathbb{N}}$ be a sequence of functions
with $F_{n}:\{0,\dots,n\}\to\mathbb{R}$ and let
$\mathcal{M}(n)\in\{0,\dots,n\}$ be a value for which the function $F_{n}$ reaches its
maximum. Assume that the sequence of functions $\{f_{n}\}_{n\in\mathbb{N}}$
defined by $f_{n}(x):=F_{n}(\lfloor nx\rfloor)$ for every $x\in[0,1]$ converges uniformly on $[0,1]$
to a continuous function $f$ and that $\theta$ is the only global maximum of
$f$ in $[0,1]$. Then,
\begin{itemize}
\item[i)] $\displaystyle\lim_{n} \mathcal{M}(n)/n =\theta$.
\item[ii)] $\displaystyle\lim_{n} F_{n}(\mathcal{M}(n))= f(\theta)$.
\end{itemize}
\end{prop}
\begin{proof}
It is identical to the proof of \cite[Proposition 1]{nuestropos}.
\end{proof}

The following result, which is rather similar to the previous one, will also turn out to be useful in the sequel.

\begin{prop}
\label{conv2} Let $\{F_{n}\}_{n\in\mathbb{N}}$ be a sequence of functions
with $F_{n}:\{0,\dots,n\}\to\mathbb{R}$ and let
$\mathcal{N}(n)\in\{0,\dots,n-1\}$ be such that
\begin{align*}
\frac{\mathcal{N}(n)}{n}&< F_{n}\left(\mathcal{N}(n) \right),\\
\frac{\mathcal{N}(n)+1}{n}&\geq F_{n}\left(\mathcal{N}(n)+1 \right).
\end{align*}
Assume that the sequence of functions $\{f_{n}\}_{n\in\mathbb{N}}$
defined by $f_{n}(x):=F_{n}(\lfloor nx\rfloor)$ for every $x\in[0,1]$ converges uniformly on $[0,1]$ to a continuous function $f$ and that $\theta$ is the only solution of $x=f(x)$. Then,
$\displaystyle\lim_{n}\mathcal{N}(n)/n=\theta$.
\end{prop}
\begin{proof}
Let us consider the sequence $\{\mathcal{N}(n)/n\}\subset[0,1]$ and
assume that $\{\mathcal{N}(s_{n})/s_{n}\}$ is a subsequence that converges to certain value
$\alpha\in[0,1]$. Then,
\begin{align*}
&\alpha=\lim_n\frac{\mathcal{N}(s_n)}{s_n}\leq \lim_n F_{s_n}(\mathcal{N}(s_n))=\lim_n F_{s_n}\left(\frac{\mathcal{N}(s_n)}{s_n}s_n\right)=\lim_n f_{s_n}\left(\frac{\mathcal{N}(s_n)}{s_n}\right)=f(\alpha),\\
&f(\alpha)=\lim_n f_{s_n}\left(\frac{\mathcal{N}(s_n)+1}{s_n}\right)=\lim_n F_{s_n}(\mathcal{N}(s_n)+1)\leq \lim_n \frac{\mathcal{N}(s_n)+1}{s_n}=\alpha.
\end{align*}
Consequently, $\alpha=f(\alpha)$ and since $\theta$ is the only solution of $x=f(x)$ it follows that
$\theta=\alpha$.

Thus, we have proved that every convergent subsequence of $\{\mathcal{N}(n)/n\}$ converges to the same limit $\theta$. Since $\{\mathcal{N}(n)/n\}$ is
defined on a compact set this implies that $\{\mathcal{N}(n)/n\}$ itself must
also converge to $\theta$.
\end{proof}

In some cases, the functions $F_n:\{0,\dots,n\}\to\mathbb{R}$ can be naturally extended to continuous functions $\widetilde{F_n}:[0,n]\to\mathbb{R}$ so that $F_n$ can be seen as the restriction of $\widetilde{F_n}$ to $[0,n]\cap\mathbb{Z}$. In this situation, $f_n(x)=F_{n}(\lfloor nx\rfloor)=\widetilde{F_n}(nx)$, and the uniform convergence of the sequence $\{f_n\}$ to a continuous function $f$ is usually easy to establish as it was the case in \cite{nuestropos} or in \cite{nuestropos2}. However, in a general situation the functions $\widetilde{F_n}$ might not be easy to find and the uniform convergence of the sequence $\{f_n\}$ to a continuous function is difficult to establish. 

In this work, we will follow an approach similar to that in \cite{yo}. Namely, we will assume the uniform convergence of the sequence $\{f_n\}$ to a continuous function. Under this assumption the following results show that the limit function $f$ can be easily found provided the functions $F_{n}$ are recursively defined. 

\begin{prop}
\label{inicial} 
Let $\{F_{n}\}_{n\in\mathbb{N}}$, $\{G_{n}\}_{n\in\mathbb{N}}$ and $\{H_{n}\}_{n\in\mathbb{N}}$ be sequences of functions with $F_n,G_n,H_n:\{0,\dots,n\}\in\mathbb{R}$ such that they satisfy
\begin{align*}
F_{n}(k)&=G_{n}(k)+H_{n}(k)F_{n}(k-1),\\ F_{n}(0)&=\mu.
\end{align*}
Moreover, for every $x\in[0,1]$, let us define $f_{n}(x):=F_{n}(\lfloor{nx}\rfloor)$, $h_{n}(x):=n(1-H_{n}(\lfloor
{nx}\rfloor))$ and $g_{n}(x):=nG_{n}(\lfloor{nx}\rfloor)$. If the following conditions hold:
\begin{itemize}
\item[i)] Both sequences $\{h_{n}\}$ and $\{g_{n}\}$ converge on $(0,1)$ and uniformly on $[\varepsilon,\varepsilon^{\prime}]$ for every $0<\varepsilon<\varepsilon^{\prime}<1$ to
continuous functions $h(x)$ and $g(x)$, respectively.
\item[ii)] The sequence $\{f_{n}\}$ converges uniformly on $[0,1]$ to a continuous function $f$. 
\end{itemize}
Then, $f(0)=\mu$ and $f$ satisfies the equation $f'=-fh+g$ in $(0,1)$.
\end{prop}
\begin{proof}
See \cite[Theorem 1]{yo}.
\end{proof}

\begin{prop}
\label{final} 
Let $\{F_{n}\}_{n\in\mathbb{N}}$, $\{G_{n}\}_{n\in\mathbb{N}}$ and $\{H_{n}\}_{n\in\mathbb{N}}$ be sequences of functions with $F_n,G_n,H_n:\{0,\dots,n\}\in\mathbb{R}$ such that they satisfy
\begin{align*}
F_{n}(k)&=G_{n}(k)+H_{n}(k)F_{n}(k+1),\\ F_{n}(n)&=\mu.
\end{align*}
Moreover, for every $x\in[0,1]$, let us define $f_{n}(x):=F_{n}(\lfloor{nx}\rfloor)$, $h_{n}(x):=n(1-H_{n}(\lfloor
{nx}\rfloor))$ and $g_{n}(x):=nG_{n}(\lfloor{nx}\rfloor)$. If the following conditions hold:
\begin{itemize}
\item[i)] Both sequences $\{h_{n}\}$ and $\{g_{n}\}$ converge on $(0,1)$ and uniformly on $[\varepsilon,\varepsilon^{\prime}]$ for every $0<\varepsilon<\varepsilon^{\prime}<1$ to
continuous functions $h(x)$ and $g(x)$, respectively.
\item[ii)] The sequence $\{f_{n}\}$ converges uniformly on $[0,1]$ to a continuous function $f$. 
\end{itemize}
Then, $f(1)=\mu$ and $f$ satisfies the equation $f'=fh-g$ in $(0,1)$.
\end{prop}
\begin{proof}
See \cite[Theorem 2]{yo}.
\end{proof}

\begin{exa}
In the case of the classic Secretary problem, the probability
of success using the threshold $k$ is given by a function $F_{n}(k)$ which satisfies the
following recurrence relation:
\begin{align*}
F_{n}(k)&=\frac{1}{n}+\frac{k}{k+1} F_{n}(k+1),\\
F_{n}(n)&=0.
\end{align*}
If we consider $G_n(k)=\frac{1}{n}$ and $H_n(k)=\frac{k}{k+1}$ we get that $g_n(x)=1$ and $h_n(x)=n\frac{\lfloor nx+1\rfloor-\lfloor nx\rfloor}{\lfloor nx+1\rfloor}$ so it is easy to check that condition i) in Proposition \ref{final} holds with $g(x)=1$ and $h(x)=\frac{1}{x}$. Consequently, if we assume condition ii); i.e., if we assume that the sequence $\{f_{n}\}$ converges uniformly on $[0,1]$ to a continuous function $f$ then this function $f$ must satisfy the ODE $f'=\frac{f}{x}-1$ and the condition $f(1)=0$.

This leads to the well-known function $f(x)=-x\log(x)$, whose maximization in $[0,1]$ together with Proposition \ref{conv}
provide the asymptotic value of the optimal threshold $n/e$ as well as the asymptotic probability of success $e^{-1}$.
\end{exa}

\section{A dynamic programming approach to the $m$-returning Secretary problem}

Let us assume that there are $n$ candidates that arrive sequentially and that there are exactly $m$ identical copies of each candidate. The order in which they are inspected is random (of course, there are $(mn)!$ possibilities). At any given step it is only possible to know who is the best candidate so far and how many copies of this candidate have been inspected. Once a candidate is accepted, the process ends and, as usual, to succeed means to select the best candidate. We seek to maximize the probability of success. 

A candidate which is better than all the preceding ones will be called a \emph{maximal candidate}. Note that it is always preferable to reject a maximal candidate on its first $m-1$ appearances since we will always be able to accept it on its $m$-th appearance and, until then, other better candidates will be called a \emph{nice candidate}. On the other hand, if we inspect a non-maximal candidate, it
is irrelevant if any other of its copies has already been inspected. In fact, once a non-maximal candidate has been inspected, all of its copies can be considered as inspected.

The previous considerations imply that, at any given step, the relevant information is just the number of different inspected objects and the number of appearances of the maximal object so far. In particular, the decision nodes of this dynamic program are the appearances of nice candidates and there are only two possible actions: accept the candidate and stop or reject it and continue. As usual, an strategy is optimal if and only if, at any decision node, we choose the action with the greatest probability of success. 

In what follows we will consider the events
\begin{align*}
X_{m,n}^{k,i}=&\textrm{``succeed following the optimal strategy after having rejected $k$ different}\\ &\textrm{candidates among which the maximal candidate has appeared $i$ times''}
\end{align*}
and we will denote $\Psi_{m,n}^{i}(k)=p\left(X_{m,n}^{k,i}\right)$. 

Recall that $\mathbf{P}^m_n$ denotes the probability of success under the optimal strategy. Hence, with our notation, we have that $\mathbf{P}^m_n=p(X_{m,n}^{1,1})=\Psi_{m,n}^{1}(1)$ (note that due to the conditions of the problem we will always reject the first candidate if $mn>1$). The following results will be devoted to provide recursive relations for the function $\Psi_{m,n}^{i}$ that will ultimately allow us to effectively compute $\mathbf{P}^m_n$. 

\begin{prop}\label{progm}
$\Psi_{m,n}^{{m}}(n)=0$ and for every $1\leq k<n$, we have that
$$\Psi_{m,n}^{{m}}(k)=\frac{k}{k+1}\Psi_{m,n}^{{m}}(k+1)+\frac{1}{k+1}\Psi_{m,n}^{1}(k+1).$$
\end{prop}
\begin{proof}
First of all, it is obvious by definition that $\Psi_{m,n}^{{m}}(n)=0$. If we have inspected $n$ different candidates (all the possible ones) and the maximal candidate has already appeared $m$ times, then this maximal candidate is in fact the best candidate and it will not appear again. Thus, we can no longer select it and we cannot succeed. 

Now, let us focus on $\Psi_{m,n}^{m}(k)$ for $1\leq k<n$. If we consider the event $A$=``The next inspected candidate is maximal'', then the Law of total probability leads to
$$\Psi_{m,n}^{{m}}(k)=p\left(X_{m,n}^{k,m}\right)=p\left(X_{m,n}^{k,m}|A\right)p(A)+p\left(X_{m,n}^{k,m}|\overline{A}\right)p(\overline{A}).$$

By the very definition it is straightforward to see that $p\left(X_{m,n}^{k,m}|A\right)=\Psi_{m,n}^{1}(k+1)$ and that $p\left(X_{m,n}^{k,m}|\overline{A}\right)=\Psi_{m,n}^{m}(k+1)$. Since, in addition, it is also obvious that $p(A)=\frac{1}{k+1}$ the result follows.
\end{proof}

\begin{prop}\label{prog1}
$\Psi_{m,n}^{m-1}(n)=1$ and for every $1\leq k<n$, we have that
\small{
\begin{align*}
\Psi_{m,n}^{m-1}(k)&=\frac{1}{mn-mk+1}\max\left\{\frac{k}{n},\Psi_{m,n}^{{m}}(k)\right\}+\frac{k(mn-mk)}{(k+1)(mn-mk+1)}\Psi_{m,n}^{m-1}(k+1)+\\
&+\frac{mn-mk}{(k+1)(mn-mk+1)}\Psi_{m,n}^1(k+1).
\end{align*}
}
\end{prop}
\begin{proof}
First of all, it is obvious by definition that $\Psi_{m,n}^{{m}-1}(n)=1$. If we have inspected $n$ different candidates (all the possible ones) and the maximal candidate has appeared $m-1$ times, then this maximal candidate is in fact the best candidate and we just have to wait until it appears again in order to guarantee the success. 

Now, let us focus on $\Psi_{m,n}^{m-1}(k)$ for $1\leq k<n$. We consider the following events:
\begin{align*}
A&=\textrm{``The next inspected object is the last copy of the maximal candidate''},\\
B&=\textrm{``The next inspected object is a new maximal candidate''},\\
C&=\textrm{``The next inspected candidate is not maximal''}.
\end{align*}
Since in this setting $\{A,B,C\}$ is a complete system of events, the law of total probability leads to
\small{$$\Psi_{m,n}^{m-1}(k)=p\left(X_{m,n}^{k,m-1}\right)=p\left(X_{m,n}^{k,m-1}|A\right)p(A)+p\left(X_{m,n}^{k,m-1}|B\right)p(B)+p\left(X_{m,n}^{k,m-1}|C\right)p(C).$$}

It is easy to see that
\begin{align*}
p(A)&=\frac{1}{mn-mk+1},\\
p(B)&=\frac{mn-mk}{(k+1)(mn-mk+1)},\\
p(C)&=1-p(A)-p(B)=\frac{k(mn-mk)}{(k+1)(mn-mk+1)}.
\end{align*}

Like in the previous proposition, the very definition leads to the fact that $p\left(X_{m,n}^{k,m-1}|B\right)=\Psi_{m,n}^1(k+1)$ and $p\left(X_{m,n}^{k,m-1}|C\right)=\Psi_{m,n}^{m-1}(k+1)$. The last remaining probability is slightly more tricky. 

We claim that $p\left(X_{m,n}^{k,m-1}|A\right)=\max\left\{\frac{k}{n},\Psi_{m,n}^{{m}}(k)\right\}$. This happens because in this case the optimal strategy will select the next candidate (which is a nice candidate) if and only if the probability of success choosing it (which is $k/n$) is greater than the probability of success if we reject the nice candidate and keep going (which is $\Psi_{m,n}^{{m}}(k+1)$). In any case, the final probability of success is the maximum of both values, as claimed.
\end{proof}

\begin{prop}\label{prog2}
Let $i\in\{1,\dots, m-2\}$. Then, $\Psi_{m,n}^{i}(n)=1$ and for every $1\leq k<n$, we have that
\small{
\begin{align*}
\Psi_{m,n}^{i}(k)&=\frac{k(mn-mk)}{(k+1)(mn-mk+m-i)}\Psi_{m,n}^{i}(k+1)+\frac{k(mn-mk)}{(k+1)(mn-mk+m-i)}\Psi_{m,n}^{i+1}(k)+\\
&+\frac{mn-mk}{(k+1)(mn-mk+m-i)}\Psi_{m,n}^1(k+1).
\end{align*}
}
\end{prop}
\begin{proof}
First of all, it is obvious by definition that $\Psi_{m,n}^{i}(n)=1$. If we have inspected $n$ different candidates (all the possible ones) and the maximal candidate has appeared $i$ times (with $1\leq i\leq m-2$), then this maximal candidate is in fact the best candidate and we just have to wait until it appears again in order to guarantee the success. 

Now, let us focus on $\Psi_{m,n}^{i}(k)$ for $1\leq k<n$. We consider the following events:
\begin{align*}
A&=\textrm{``The next inspected object is another copy of the maximal candidate''},\\
B&=\textrm{``The next inspected object is a new maximal candidate''},\\
C&=\textrm{``The next inspected candidate is not maximal''}.
\end{align*}
Since in this setting $\{A,B,C\}$ is a complete system of events, the law of total probability leads to
\small{$$\Psi_{m,n}^{i}(k)=p\left(X_{m,n}^{k,i}\right)=p\left(X_{m,n}^{k,i}|A\right)p(A)+p\left(X_{m,n}^{k,i}|B\right)p(B)+p\left(X_{m,n}^{k,i}|C\right)p(C).$$}

It is easy to see that
\begin{align*}
p(A)&=\frac{m-i}{mn-mk+m-i},\\
p(B)&=\frac{mn-mk}{(k+1)(mn-mk+m-i)},\\
p(C)&=1-p(A)-p(B)=\frac{k(mn-mk)}{(k+1)(mn-mk+m-i)}.
\end{align*}

Like before, the very definition leads to the fact that $p\left(X_{m,n}^{k,i}|A\right)=\Psi_{m,n}^{i+1}(k)$, $p\left(X_{m,n}^{k,i}|B\right)=\Psi_{m,n}^1(k+1)$ and $p\left(X_{m,n}^{k,i}|C\right)=\Psi_{m,n}^{i}(k+1)$ and the result follows.
\end{proof}

Once we have established the previous recursive relations, we are in condition to determine the computational complexity of the associated dynamic program.

\begin{prop} 
For any fixed $m$, the computational complexity of the dynamic program defined by the previous recurrences is $O(n)$.
\end{prop}
\begin{proof} 
It is enough to observe that the number of required operations to compute $\{\Psi_{m,n}^{i}(k)\}_{i=1}^m$ from $\{\Psi_{m,n}^{i}(k+1)\}_{i=1}^m$ is independent of $n$.
\end{proof}

We have already mentioned that if a maximal candidate is accepted after the inspection of $k$ different candidates, the probability of success if $k/n$, just like in the classical Secretary problem. Also recall that it is always preferable to reject a maximal candidate unless it is a nice candidate. Consequently, it is clear that the optimal strategy consists in accepting a nice candidate whenever the number of different inspected candidates belongs to the so-called stopping set
$$\mathcal{S}:=\{k:k/n\geq \Psi_{m,n}^m(k)\}.$$
If the stopping set consists of a single stopping island (see \cite{sonin} for a precise definition) then we say that the optimal strategy is a \emph{threshold strategy}. In such a case, $\min \mathcal{S}$is called the \emph{optimal threshold}. We now prove that like in the classic Secretary problem, the optimal strategy for the $m$-returning Secretary problem is a threshold strategy. Although this was already proved by Garrod, we provide a simpler proof based in the previous dynamic program.

\begin{teor}
\label{BWS} 
In the $m$-returning Secretary problem, let $n$ be the number of
different objects. Then, there exists $\mathbf{k}_{n}^{m}$ such that the following
strategy is optimal:
\begin{enumerate}
\item Reject the $\mathbf{k}_{n}^{m}$ first different inspected objects.
\item After that, accept the first nice candidate.
\end{enumerate}
\end{teor}
\begin{proof}
We have to prove that, for every $k$ the following holds
$$\frac{k}{n}\geq \Psi_{m,n}^m(k) \Longrightarrow \frac{k+1}{n} \geq \Psi_{m,n}^m(k+1).$$

To do so, it suffices to see that $\Psi_{m,n}^m$ is non-increasing. In fact, the very definition implies that $\Psi_{m,n}^{m}(k+1)\leq\Psi_{m,n}^{1}(k+1)$ so, applying Proposition \ref{progm} we get that
$$\Psi_n^m(k)=\frac{k}{1+k}\,\Psi_{m,n}^{{m}}(k+1)+\frac{1\,}{1+k}\Psi_{m,n}^{1}(k+1)\geq \Psi_{m,n}^{{m}}(k+1)$$
and the result follows.
\end{proof}

\section{The asymptotic optimal stopping threshold}

In the previous section we have proved the existence of an optimal stopping threshold $k_n^m$. Now, we will study its asymptotic behavior. Namely, we will compute $\lim_{n\to\infty}\mathbf{k}_n^m/n$. To do so, we need to consider the events
\begin{align*}
Y_{m,n}^{k,i}=&\textrm{``succeed accepting the first nice candidate after having rejected $k$ different}\\ &\textrm{candidates among which the maximal candidate has appeared $i$ times''}
\end{align*}
and we will denote $\Phi_{m,n}^{i}(k)=p\left(X_{m,n}^{k,i}\right)$. 

These functions $\Phi_{m,n}^{i}$ satisfy nearly the same recursive relations that were satisfied by the functions $\Psi_{m,n}^{i}$ as we see in the proposition below.

\begin{prop}\label{recphi}
$\Phi_{m,n}^{i}(n)=0$ for every $i\in\{1,\dots,m-1\}$, $\Phi_{m,n}^{m}(n)=1$ and, if $1\leq k<n$, the following hold
\small{
\begin{align*}
\Phi_{m,n}^{{m}}(k)=&\frac{k}{k+1}\Phi_{m,n}^{{m}}(k+1)+\frac{1}{k+1}\Phi_{m,n}^{1}(k+1).\\
\Phi_{m,n}^{m-1}(k)=&\frac{1}{mn-mk+1}\Phi_{m,n}^{{m}}(k)+\frac{k(mn-mk)}{(k+1)(mn-mk+1)}\Phi_{m,n}^{m-1}(k+1)+\\
&+\frac{mn-mk}{(k+1)(mn-mk+1)}\Phi_{m,n}^1(k+1).\\
\Phi_{m,n}^{i}(k)=&\frac{k(mn-mk)}{(k+1)(mn-mk+m-i)}\Phi_{m,n}^{i}(k+1)+\frac{k(mn-mk)}{(k+1)(mn-mk+m-i)}\Phi_{m,n}^{i+1}(k)+\\
&+\frac{mn-mk}{(k+1)(mn-mk+m-i)}\Phi_{m,n}^1(k+1).
\end{align*}
}
\end{prop}
\begin{proof}
Just reason in the same way as in Propositions \ref{progm}, \ref{prog1} and \ref{prog2}.
\end{proof}

As a consequence of their very similar definition and since they satisfy nearly the same recursive relations, this is no surprise that the functions $\Phi_{m,n}^{i}$ and $\Psi_{m,n}^{i}$ are closely related. In fact, we have the following result.

\begin{prop}\label{phipsi}
The following relations hold.
\begin{itemize}
\item[i)] If $k\geq\mathbf{k}_n^m$, then $\Phi_{m,n}^{i}(k)=\Psi_{m,n}^{i}(k)$ for every $1\leq i\leq m$.
\item[ii)] $\Phi_{m,n}^{i}(\mathbf{k}_n^m-1)\leq\Psi_{m,n}^{i}(\mathbf{k}_n^m-1)$ for every $1\leq i\leq m-1$.
\item[iii)] $\Phi_{m,n}^{m}(\mathbf{k}_n^m-1)=\Psi_{m,n}^{m}(\mathbf{k}_n^m-1)$.
\item[iv)] If $1\leq k<\mathbf{k}_n^m$, then $\Phi_{m,n}^{i}(k)\leq\Psi_{m,n}^{i}(k)$ for every $1\leq i\leq m$.
\end{itemize}
\end{prop}
\begin{proof} 
To prove i) it is enough to recall that $\mathbf{k}_n^m$ is the optimal threshold. Hence, if $k\geq \mathbf{k}_n^m$, the optimal strategy (recall Theorem \ref{BWS}) that defines $\Psi_{m,n}^{i}(k)$ coincides with the strategy that defines $\Phi_{m,n}^{i}(k)$ and both functions are equal as claimed.

Now, to prove ii), iii) and iv) it is enough to apply i) and Propositions \ref{progm}, \ref{prog1}, \ref{prog2} and \ref{recphi} an then proceed inductively. 
\end{proof}

\begin{prop}\label{desigphi}
Let $\mathbf{k}_n^m$ be the optimal threshold. Then, $\Phi_{m,n}^{{m}}(\mathbf{k}_n^m)\leq \mathbf{k}_n^m/n$ and $\Phi_{m,n}^{{m}}(\mathbf{k}_n^m-1)>(\mathbf{k}_n^m-1)/n$.
\end{prop}
\begin{proof}
Recall that, by definition $\mathbf{k}_n^m=\min\{k:k/n\geq \Psi_{m,n}^m(k)\}$. Thus, it is enough to apply this fact and Proposition \ref{phipsi} i) and iv).
\end{proof}

Now, for every $1\leq i\leq m$, let us define $\phi_{m,n}^{i}(x):=\Phi_{m,n}^{i}(\lfloor xn\rfloor)$. Under suitable assumptions over the uniform convergence of the sequence $\{\phi_{m,n}^{i}\}_n$, we will be able to apply Proposition \ref{conv2} and Proposition \ref{final} in order to determine the asymptotic behavior of $\mathbf{k}_n^m$.

AQUI UNA GRAFICA

\begin{prop}
\label{sistema}
Let us assume that, for every $1\leq i\leq m$, the sequence of functions $\{\phi_{m,n}^{i}\}_n$ converges uniformly on $[0,1]$
to a function $y_{i}$ which is continuous and derivable on $(0,1]$. Then, the functions $y_i$ satisfy the following system of ODEs on the interval $(0,1]$
with $y_{i}(1)=1$ for every $1\leq i\leq m-1$ and $y_m(1)=0$.
$$\begin{cases}
(m-mx)y_1'(x)=(m-1)y_1(x)-(m-1)y_2(x)\\
(mx-mx^2)y'_i(x)=m(x-1)y_1(x)+(m-ix)y_i(x)+x(m-1)y_{i+1}(x),\ 2\leq i\leq m-2\\
(mx-mx^2)y'_{m-1}(x)=-(m-mx)y_1(x)+(m-(m-1)x)y_{m-1}(x)-x^2\\
xy'_m(x)=-y_1(x)+y_m(x)
\end{cases}
$$
\end{prop}
\begin{proof}
Due to Proposition \ref{recphi}, we have that $\Phi_{n}^{1}(n)=1$ and
$$\Phi_{m,n}^{1}(k)=G_{n}(k)+H_{n}(k)\Phi_{m,n}^{1}(k+1),$$
where
\begin{align*}
G_{n}(k)&=\frac{{m}-1}{\,{m}\text{ }n-{m}\text{ }k+{m}-1}\Phi_{m,n}^{2}(k),\\
H_{n}(k)&=\frac{{m}\text{ }n-{m}\text{ }k}{\,{m}\text{ }n-{m}\text{ }k+{m}-1}.
\end{align*}

If we denote $h_{n}(x):=n(1-H_{n}(\lfloor{nx}\rfloor))$ and $g_{n}(x):=nG_{n}(\lfloor{nx}\rfloor)$, we are in the conditions to apply Proposition \ref{final} with 
$\displaystyle g(x)=\frac{m-1}{m-mx}y_{2}(x)$ and $\displaystyle h(x)=\frac{m-1}{m-mx}$ and we can conclude that $y_1(1)=1$ and that $y_{1}(x)$ satisfies the following ODE 
$$y'_{1}(x)=h(x)y_{1}(x)-g(x)=\frac{m-1}{m-mx}y_{1}(x)-\frac{m-1}{m-mx}y_{2}(x).$$

The remaining equations arise int he same way just using Proposition \ref{recphi} and Proposition \ref{final} repeatedly.
\end{proof}

If we denote the solutions of the system of ODEs from Proposition \ref{sistema} by $Y_m^i(x)$ for every $i\leq 1\leq m$, we have that $\displaystyle \phi_{m,n}^i\to Y_m^i$ uniformly on $[0,1]$ and we have the following result. 

\begin{teor}\label{asin}
Let us assume that, for every $1\leq i\leq m$, the sequence of functions $\{\phi_{m,n}^{i}\}_n$ converges uniformly on $[0,1]$. Also, let $\mathbf{k}_n^m$ be the optimal threshold. Then, $\lim_{n}\mathbf{k}_{n}^m/n=\vartheta_m$, where $\vartheta_m$ is the solution to the equation $x=Y_m^{m}(x)$.
\end{teor}
\begin{proof} 
Proposition \ref{sistema} and Proposition \ref{desigphi} imply that we are in the conditions to apply Proposition \ref{conv2} and the result follows immediately.
\end{proof}

Theorem \ref{asin} can be used to compute $\lim_{n}\mathbf{k}_{n}^m/n$ with arbitrary precision because, due to Proposition \ref{sistema}, we can obtain $Y_m^m(x)$ as a power series centered at $x=1$. As an example we work out the case $m=3$ but the reasoning would be essentially the same for any other value of $m$.

\begin{cor} 
With the previous notation, 
$$\lim_{n}\mathbf{k}_{n}^3/n=\vartheta_3=0.49263576026053198177870853577593\dots$$
\end{cor}
\begin{proof} 
For the sake of simplicity, let us denote $f(x)=Y_3 ^3(x)$. First of all observe (we omit the details) that, for every $i>1$,
$$\left |\frac{f^{(i+1)}(1)}{f^{(i)}(1)}\right|<i.$$

Consequently, the following power series has radius of convergence greater or equal than 1:
$$f(x)=\sum_{i=1}^{\infty}\frac{(x-1)^{i}f^{(i)}(1)}{i!}.$$

Now, consider the truncated series
$$\overline{f}(x)=\sum_{i=1}^{1000}\frac{(x-1)^{i}f^{(i)}(1)}{i!}.$$
Since $\left|\frac{f^{(i+1)}(1)}{f^{(i)}(1)}\right|<i$ and $f^{\prime}(1)=-1$, it follows that $|f^{(i)}(1)|<i!$ so
$$|f(x)-\overline{f}(x)|=\left|\sum_{i=1001}^{\infty}\frac{(x-1)^{i}f^{(i)}(1)}
{i!}\right|<\left|\sum_{i=1001}^{\infty}(1-x)^{i}\right|=\frac{\left(1-x\right)^1001}{x}.$$
Thus, for every $x\in [1/4,1]$, we have that $|f(x)-\overline{f}(x)|<4\cdot10^{-124}$ and in the exact same way we obtain that $|f'(x)-\overline{f}'(x)|<4\cdot10^{-124}$.

Let us denote $\overline{\vartheta}$ such that $\overline{f}(\overline{\vartheta
})=\overline{\vartheta}$ and recall that $\vartheta_3$ satisfies that $f(\vartheta_3)=(\vartheta_3)$. Now, Lagrange's mean value theorem implies that, for some $c$ between $\vartheta$ and $\vartheta_3$ it holds that
$$|\vartheta_3-\overline{\vartheta}|=\left|\frac{\overline{f}(\vartheta_3)-f(\vartheta_3)}{\overline{f}^{\prime}(c)-1}\right|.$$
Since we can choose $c$ such that $\\overline{f}^{\prime}(c)-1|>1$, it follows that $|\vartheta_3-\overline{\vartheta}|<|\overline{f}(\vartheta_3)-f(\vartheta_3)|<4\cdot10^{-124}$. But we can compute 
$\overline{\vartheta}$ with arbitrary precision so if we consider, for example,
$$\overline{\vartheta}=0.49263576026053198177870853577593\dots$$
we are done.
\end{proof}

\section{Probability of success}

In order to compute the probability of success $\mathbf{P}_n^m$, we first consider the following event:
\begin{align*}
Y_{m,n}^{k}=&\textrm{``succeed accepting the first nice candidate after having rejected}\\ &\textrm{$k$ different candidates''}
\end{align*}
and we denote $\mathbf{P}_n^m(k)=p\left(Y_{m,n}^{k}\right)$. Clearly, $\mathbf{P}_n^m=\mathbf{P}_n^m\left(\mathbf{k}_n^m\right)$.

Moreover, let us define now the following events:
\begin{align*}
Z_{m,n}^{i,k}=&\textrm{``the maximal candidate has been inspected $i$ times when}\\ &\textrm{the $k$-th different candidate is inspected''}
\end{align*}
and define $\Theta_{m,n}^{i}(k)=p\left(Z_{m,n}^{i,k}\right)$.

Clearly, the family $\{Z_{m,n}^{i,k}\}_{i=1}^m$ is a complete system of events. Thus, the law of total probability leads to
\begin{align*}
\mathbf{P}_n^m(k)&=p\left(Y_{m,n}^{k}\right)=\sum_{i=1}^mp\left(Y_{m,n}^{k}|Z_{m,n}^{i,k}\right)p\left(Z_{m,n}^{i,k}\right)=\\
&=\sum_{i=1}^m p\left(Y_{m,n}^{i,k}\right)p\left(Z_{m,n}^{i,k}\right)=\sum_{i=1}^m \Phi_{m,n}^i(k)\Theta_{m,n}^{i}(k).
\end{align*}

Consequently, in order to determine $\mathbf{P}_n^m(k)$ we need to determine the value of $\Theta_{m,n}^{i}(k)$ for every $1\leq i\leq m$. We do so in the following series of results.

\begin{prop}
\label{z1}
$\Theta_{m,n}^1(1)=1$ and for every $1<k\leq n$, we have that
\small{
$$\Theta_{m,n}^{1}(k)=\frac{1}{k}+\left(\frac{m-1}{k(mn-m(k-1)+m-1)}+\frac{mn-m(k-1)}{mn-m(k-1)+m-1}-\frac{1}{k}\right)\Theta_{m,n}^{1}(k-1).$$}
\end{prop}
\begin{proof}
First of all, $\Theta_{m,n}^1(1)=1$ by definition because if only one different candidate has been inspected, it is obviously maximal.

Now, as we have already done before, we apply the law of total probability to get that
$$
\Theta_{m,n}^{1}(k)=p\left(Z_{m,n}^{1,k}\right)=p\left(Z_{m,n}^{1,k}|Z_{m,n}^{1,k-1}\right)p\left(Z_{m,n}^{1,k-1}\right)+p\left(Z_{m,n}^{1,k}|\overline{Z_{m,n}^{1,k-1}}\right)p\left(\overline{Z_{m,n}^{1,k-1}}\right).
$$
Since $p\left(Z_{m,n}^{1,k-1}\right)=1-p\left(\overline{Z_{m,n}^{1,k-1}}\right)=\Theta_{m,n}^{1}(k-1)$, we just have to compute the remaining terms:
\begin{itemize}
\item To compute $p\left(Z_{m,n}^{1,k}|\overline{Z_{m,n}^{1,k-1}}\right)$, we assume that the maximal candidate has appeared more than once when the $(k-1)$-th different candidate is inspected. Then, the only possible way in which the maximal candidate can appear once when the $k$-th different candidate is inspected is that this $k$-th candidate is in fact a maximal candidate. Since this happens with probability $1/k$, we have just seen that 
$$p\left(Z_{m,n}^{1,k}|\overline{Z_{m,n}^{1,k-1}}\right)=\frac{1}{k}.$$
\item To compute $p\left(Z_{m,n}^{1,k}|Z_{m,n}^{1,k-1}\right)$, we assume that the maximal candidate has appeared once when the $(k-1)$-th different candidate is inspected. Then, there are two possible ways in which the maximal candidate can appear once when the $k$-th different candidate is inspected:
\begin{itemize}
\item The maximal element appears again before the inspection of the $k$-th different candidate and the $k$-th different candidate is a new maximal. This happens with probability
$$\frac{m-1}{mn-m(k-1)+m-1}\cdot\frac{1}{k}.$$
\item The maximal element does not appear again before the inspection of the $k$-th different candidate and the $k$-th different candidate is irrelevant. This happens with probability
$$1-\frac{m-1}{mn-m(k-1)+m-1}=\frac{mn-m(k-1)}{mn-m(k-1)+m-1}.$$
\end{itemize}
Consequently, we have just seen that 
$$p\left(Z_{m,n}^{1,k}|Z_{m,n}^{1,k-1}\right)=\left(\frac{m-1}{k(mn-m(k-1)+m-1)}\frac{mn-m(k-1)}{mn-m(k-1)+m-1}\right),$$
and it is enough to combine all the previous calculation to get the result.
\end{itemize}
\end{proof}

\begin{prop}
\label{z2}
$\Theta_{m,n}^2(1)=0$ and for every $1<k\leq n$, we have that
\begin{align*}
\Theta_{m,n}^{2}(k)=&\frac{(k-1)(mn-m(k-1))}{k(mn-m(k-1)+m-2)}\Theta_{m,n}^{2}(k-1)+\\
&+\frac{\left(k-1\right)(m-1)(mn-m(k-1))}{k\left(mn-m(k-1)+m-2\right)(mn-m(k-1)+m-1)}\Theta_{m,n}^{1}(k-1)
\end{align*}
\end{prop}
\begin{proof}
First of all, $\Theta_{m,n}^2(1)=0$ by definition because if only one different candidate has been inspected, it is impossible that the maximal element has appeared twice.

Note that the family $\{Z_{m,n}^{i,k-1}\}_{i=1}^{k-1}$ is a complete system of events. Thus, the law of total probability leads to
$$
\Theta_{m,n}^{2}(k)=p\left(Z_{m,n}^{2,k}\right)=\sum_{i=1}^{k-1}p\left(Z_{m,n}^{2,k}|Z_{m,n}^{i,k-1}\right)p\left(Z_{m,n}^{i,k-1}\right).
$$

Now, we have the following:
\begin{itemize}
\item If $i\geq 3$, $p\left(Z_{m,n}^{2,k}|Z_{m,n}^{i,k-1}\right)=0$ by definition. Every time we inspect a new different candidate the number of times the maximal candidate has appeared either stays the same (if no new copies of the maximal candidate appear in between and the new candidate is not maximal), increases (if new copies of the maximal candidate appear in between and the new candidate is not maximal) or decreases to 1 (if the new candidate is a new maximal).
\item To compute $p\left(Z_{m,n}^{2,k}|Z_{m,n}^{2,k-1}\right)$, we assume that the maximal candidate has appeared twice when the $(k-1)$-th different candidate is inspected. Then, the only possible way in which the maximal candidate can appear twice when the $k$-th different candidate is inspected is that the maximal element does not appear again before the inspection of the $k$-th different candidate and the $k$-th different candidate is not a new maximal. This happens with probability
$$\frac{mn-m(k-1)}{mn-m(k-1)+m-2}\cdot\frac{k-1}{k}.$$
\item To compute $p\left(Z_{m,n}^{2,k}|Z_{m,n}^{1,k-1}\right)$, we assume that the maximal candidate has appeared once when the $(k-1)$-th different candidate is inspected. Then, the only possible way in which the maximal candidate can appear twice when the $k$-th different candidate is inspected is that the maximal element appears only once before the inspection of the $k$-th different candidate and the $k$-th different candidate is not a new maximal. This happens with probability
$$\frac{m-1}{mn-m(k-1)+m-1}\cdot\frac{mn-m(k-1)}{mn-m(k-1)+m-2}\cdot\frac{k-1}{k}.$$
\end{itemize}
Since by definition $p\left(Z_{m,n}^{i,k-1}\right)=\Theta_{m,n}^{i}(k-1)$, the result follows immediately.
\end{proof}

\begin{prop}
\label{zi}
Let $i\geq 3$. Then, $\Theta_{m,n}^{i}(1)=0$ and for every $1<k\leq n$, we have that
\small{
$$\Theta_{m,n}^{i}(k)=\frac{k-1}{k}\frac{mn-m(k-1)}{mn-m(k-1)+m-i}\Theta_{m,n}^{i}(k-1)+\frac{m-i+1}{mn-m(k-1)+m-i}\Theta_{m,n}^{i-1}(k)$$}
\end{prop}
\begin{proof}
First of all, if $i\geq 3$, $\Theta_{m,n}^i(1)=0$ by definition because if only one different candidate has been inspected, it is impossible that the maximal element has appeared more than twice.

Note that the family $\{Z_{m,n}^{j,k-1}\}_{j=1}^{k-1}$ is a complete system of events. Thus, the law of total probability leads to
$$
\Theta_{m,n}^{i}(k)=p\left(Z_{m,n}^{i,k}\right)=\sum_{j=1}^{k-1}p\left(Z_{m,n}^{i,k}|Z_{m,n}^{j,k-1}\right)p\left(Z_{m,n}^{j,k-1}\right).
$$

Now, we have the following:
\begin{itemize}
\item If $j\geq i+1$, $p\left(Z_{m,n}^{2,k}|Z_{m,n}^{j,k-1}\right)=0$ by definition. Every time we inspect a new different candidate the number of times the maximal candidate has appeared either stays the same (if no new copies of the maximal candidate appear in between and the new candidate is not maximal), increases (if new copies of the maximal candidate appear in between and the new candidate is not maximal) or decreases to 1 (if the new candidate is a new maximal).
\item To compute $p\left(Z_{m,n}^{i,k}|Z_{m,n}^{i,k-1}\right)$, we assume that the maximal candidate has appeared $i$ times when the $(k-1)$-th different candidate is inspected. Then, the only possible way in which the maximal candidate can appear $i$ times when the $k$-th different candidate is inspected is that the maximal element does not appear again before the inspection of the $k$-th different candidate and the $k$-th different candidate is not a new maximal. This happens with probability
$$\frac{mn-m(k-1)}{mn-m(k-1)+m-i}\cdot\frac{k-1}{k}.$$
\item To compute $p\left(Z_{m,n}^{i,k}|Z_{m,n}^{j,k-1}\right)$ for $1\leq j\leq i-1$, we assume that the maximal candidate has appeared $j$ times when the $(k-1)$-th different candidate is inspected. Then, the only possible way in which the maximal candidate can appear $i$ times when the $k$-th different candidate is inspected is that the maximal element appears exactly $i-j$ times again before the inspection of the $k$-th different candidate and the $k$-th different candidate is not a new maximal. This happens with probability
$$\left(\frac{mn-m(k-1)}{mn-m(k-1)+m-i}\cdot\prod_{l=j}^{i-1}\frac{m-l}{mn-m(l-1)+m-l}\right)\cdot\frac{k-1}{k}.$$
\end{itemize}

By definition $p\left(Z_{m,n}^{j,k-1}\right)=\Theta_{m,n}^{j}(k-1)$. Moreover, some elementary computations show that 
\small{
\begin{align*}
&\sum_{j=1}^{i-1}p\left(Z_{m,n}^{i,k}|Z_{m,n}^{j,k-1}\right)\Theta_{m,n}^{j}(k-1)=\\
&=\sum_{j=1}^{i-1}\left(\frac{mn-m(k-1)}{mn-m(k-1)+m-i}\cdot\prod_{l=j}^{i-1}\frac{m-l}{mn-m(l-1)+m-l}\right)\cdot\frac{k-1}{k}\Theta_{m,n}^{j}(k-1)=\\
&=\frac{m-i+1}{mn-m(k-1)+m-i}\Theta_{m,n}^{i-1}(k)
\end{align*}
}
and the result follows.
\end{proof}

Now, for every $1\leq i\leq m$, let us define $\theta_{m,n}^{i}(x):=\Theta_{m,n}^{i}(\lfloor xn\rfloor)$. Under suitable assumptions over the uniform convergence of the sequence $\{\theta_{m,n}^{i}\}_n$, we will be able to apply Proposition \ref{inicial}.

\begin{prop}\label{limz}
Let us assume that, for every $1\leq i\leq m$, the sequence of functions $\{\theta_{m,n}^{i}\}_n$ converges uniformly on $[0,1]$
to a function $z_{i}$ which is continuous and derivable on $[0,1)$. Then,
\begin{align*}
z_m^{1}(x)  &  =
\begin{cases}
\frac{{m}\,{\left(  -1+(1-x)^{-\frac{1}{{m}}}\right)  }(\,1-x)}{x}; & \textrm{if 
$x\in(0,1]$},\\
1; & \textrm{if $x=0$}.
\end{cases}
\\
z_{m}^{2}(x)  &  =
\begin{cases}
\frac{{m}({m}-1){\left(  (-1+(1-x)^{\frac{1}{{m}}}\right)  }^{2}(1-x)^{\frac{{m}-2}%
{{m}}}}{2x}; & \textrm{if 
$x\in(0,1]$},\\
0; & \textrm{if $x=0$}.
\end{cases}
\\
z_{m}^{i}(x)  &  =
\begin{cases}
\frac{{m}!\,{\left(  (-1)^{i}\left(  -1+(1-x)^{\frac{1}{{m}}}\right)  \right)
}^{i}(1-x)^{\frac{{m}-i}{{m}}}}{i!\text{ }x\text{ }({m}-i)!}; & \textrm{if 
$x\in(0,1]$},\\
0; & \textrm{if $x=0$}.
\end{cases}\ \textrm{For every $3\leq i\leq m$.}
\end{align*}
\end{prop}
\begin{proof}
Taking into account Proposition \ref{z1}, we have that $\Theta_{m,n}^{1}(1)=1$ and
$$\Theta_{m,n}^{1}(k)=G_{n}(k)+H_{n}(k)\Theta_{m,n}^{1}(k-1),$$
where
$$G_{n}(k)=\frac{1}{k},$$
$$H_{n}(k)=\frac{m-1}{k(mn-m(k-1)+m-1)}+\frac{mn-m(k-1)}{mn-m(k-1)+m-1}-\frac{1}{k}.$$

Thus, if we denote $h_{n}(x)=n(1-H_{n}(\lfloor{nx}\rfloor))$ and $g_{n}(x):=nG_{n}(\lfloor{nx}\rfloor)$, we are in the conditions to apply Proposition \ref{inicial} with $g(x)=\frac{1}{x}$ and 
$h(x)=\frac{m-x}{mx-x^{2}}$ to conclude that $z_m^1$ satisfies the following ODE on $(0,1)$:
$$y'(x)=g(x)-h(x)y(x)=\frac{1}{x}-\frac{m-x}{mx-mx^{2}}y(x).$$

Since $z_m^{1}(x)$ is continuous at $x=0$ with $z_m^{1}(0)=1$, it is enough to solve this ODE to get that
$$\Theta_m^{1}(x)=\frac{m\left(-1+(1-x)^{-\frac{1}{m}}\right)(1-x)}{x}$$
as claimed.

Finally, the remaining cases can we worked out in the exact same way just considering the recursive relations and initial conditions from Proposition \ref{z2} and Proposition \ref{zi}, respectively.
\end{proof}

Recall that $\mathbf{P}_n^m(k)=\sum_{i=1}^m \Phi_{m,n}^i(k)\Theta_{m,n}^{i}(k)$. If we now define $\mathbf{p}_n^m(x):=\mathbf{P}_n^m\left(\lfloor nx\rfloor\right)$, the following result is straightforward.

\begin{cor}
Under the assumptions from Proposition \ref{sistema} and Proposition \ref{limz}, the sequence of functions $\{\mathbf{p}_n^m\}_n$ converges uniformly on $[0,1]$ to the function
$$\pi_m(x):=\sum_{i=1}^{m}Y_m^{i}(x)z_m^{i}(x).$$
\end{cor}

This corollary leads immediately to the final result of the paper, that allows us to determine the asymptotic probability of success. Recall that $\vartheta_m$ is the solution to the equation $Y_m^m(x)=x$.

\begin{teor}
Under the suitable assumptions about uniform convergence, 
$$\lim_{n}\mathbf{P}_{n}^m=\pi_m(\vartheta_m).$$
\end{teor}
\begin{proof}
We need only consider Theorem \ref{conv} and that $\vartheta_m$ is the value at which $\pi_m$ reaches its maximum value on $[0,1]$.
\end{proof}

\section{Final comments}

To close the paper, we are to make some comments relating our results to the five open problems proposed by Garrod that were presented in the introduction.
\begin{itemize}
\item The first open problem asked for an improvement on the formula to compute $\mathbf{P}_n^m$. Even if we have not given a closed formula, in Section 3 we have obtained a recursive formula that
allows its computation in linear time (with respect to the number of objects).
\item Open problems 2 and 3 have been partially solved. We have no doubt that both limits exist, but our proofs are conditioned to assume the uniform convergence of certain sequences of functions. Nevertheless, under these assumptions, we have provided a method to compute them with arbitrary precision.
\item Finally, regarding open problems 4 and 5, we have been able to find the value of $\lim_n(\mathbf{k}_n^m/n)$ as the root of the equation $Y_m^m(x)=x$. This root can be approximated using a power-series whose coefficients can be computed with arbitrary precision. however, it cannot be expressed in terms of elementary or special functions. In addition, $\lim_n(\mathbf{P}_n^m)$ can be computed just using the fact that $\lim_n(\mathbf{P}_n^m)=\pi_m(\vartheta_m)$.
\end{itemize}

Finally, we provide the following table showing the optimal threshold $\mathbf{k}_n^m$ and the probability of success $\mathbf{P}_n^m$ for some values of $n$ and $1\leq m\leq 10$. We also provide the asymtotic values computed using Taylor series of the appropriate degree. It is worth comparing these table with Table \ref{tablagarrod}.

\begin{center}
\begin{table}[h]
\caption{Optimal threshold, probability of success and asymptotic values using our results.}
\begin{tabular}
[c]{|l|l|l|l|l|l|l|l|l|}\hline
$m$ & $\mathbf{k}_{100}^m$ & $\mathbf{k}_{1000}^{m}$ & $\mathbf{k}_{10000}^{m}$ &
$\lim_{n}\left(  \frac{\mathbf{k}_{n}^{m}}{n}\right)  $ & $\mathbf{P}_{100}^{m}$ &
$\mathbf{P}_{1000}^{m}$ & $\mathbf{P}_{10000}^{m}$ & $\lim_{n}\mathbf{P}_{n}^{m}$\\\hline
1 & 38 & 369 & 3679 & 0.367879441 & 0.37104277 & 0.36819561   & 0.36791104 & 0.3678794 \\\hline
2 & 48 &  471 & 4710  & 0.470926543 & 0.76970661 & 0.76814759      & 0.76799160  & 0.7679742\\\hline
3 &50  & 493 &  4927 & 0.492635760  & 0.93518916     & 0.93490075  & 0.93487222  & 0.9348690\\\hline
4 & 50 & 499 & 4981  & 0.498053032   & 0.93490075    & 0.98307710  & 0.98307411     & 0.9830737 \\\hline
5 & 50 & 500 & 4995  & 0.499479760  & 0.99561947    & 0.99561715   & 0.99561693   & 0.9956169\\\hline
6 & 50 & 500 & 4999  & 0.499861014   & 0.99885461    & 0.99885447  & 0.99885446   & 0.9988544 \\\hline
7 & 50 & 500 & 5000  & 0.499963006  & 0.99969900     & 0.99969899  & 0.99969899   & 0.9996989 \\\hline
8 & 50 &500  &  5000 & 0.499990198     & 0.99992082       & 0.99992082 &   0.99992082   & 0.9999208 \\\hline
9 & 50  &500  & 5000 & 0.499997415  & 0.99997920  & 0.99997920 &  0.99997920& 0.9999792 \\\hline
10 & 50  &500 & 5000 & 0.499999321  & 0.99999455   & 0.99999455  &0.99999455   & 0.9999945 \\\hline
\end{tabular}
\end{table}
\end{center}

\end{document}